\documentclass[11pt]{amsart}
\usepackage{amscd,amssymb,amsmath,enumerate}
\usepackage{amsmath,amsthm}
\usepackage[all]{xy}
\usepackage{graphicx}
\usepackage{graphics}
\usepackage{latexsym}
\usepackage[all]{xy}
\usepackage{psfrag}
\xyoption{matrix} \xyoption{arrow}
\usepackage[normalem]{ulem}

\newtheorem{proposition}{Proposition}[section]
\newtheorem{theorem}[proposition]{Theorem}
\newtheorem{lemma}[proposition]{Lemma}
\newtheorem{corollary}[proposition]{Corollary}

\newtheorem{definition}[proposition]{{Definition}}
\newenvironment{defn}{\begin{definition} \rm}{\end{definition}}

\newtheorem{remark}[proposition]{{Remark}}

\newtheorem{Example}[proposition]{Example}

\newcommand{\cB}{{\mathcal B}}

\newcommand{\cN}{{\mathcal N}}

\newcommand{\cT}{{\mathcal T}}
\newcommand{\cG}{{\mathcal G}}
\newcommand{\cE}{{KQeKQ}}

\newcommand{\ma}{{\mathfrak a}}

\renewcommand{\dim}{\operatorname{dim}\nolimits}

\newcommand{\grb}{Gr\"obner\ }
\newcommand{\tip}{\operatorname{tip}}

\newcommand{\Span}{\operatorname{Span}}
\newcommand{\Imo}{I_{Mon}}

\newcommand{\extto}{\xrightarrow}

\newcommand{\wh}[1]{\widehat{v_1+ \cdots+ v_{#1}}}

\newcommand{\whe}{_{\widehat{e}}}

\usepackage{color}
\definecolor{candyapplered}{rgb}{1.0, 0.03, 0.0}
\begin{document}

\date{today}
\title[On Quasi-hereditary algebras]
{On quasi-hereditary algebras}

\author[Green]{Edward L.\ Green}
\address{Edward L.\ Green, Department of
Mathematics\\ Virginia Tech\\ Blacksburg, VA 24061\\
USA}
\email{green@math.vt.edu}
\author[Schroll]{Sibylle Schroll}
\address{Sibylle Schroll\\
Department of Mathematics \\
University of Leicester \\
University Road  \\
Leicester LE1 7RH, UK
}
\email{schroll@le.ac.uk}

\thanks{The  authors were partially supported by an LMS scheme 4 grant. The second author is supported by the EPSRC through an Early Career Fellowship EP/P016294/1}

\begin{abstract}
 Establishing
whether an algebra is quasi-hereditary or not is, in general, a difficult problem. In this paper we introduce a sufficient criterion to determine whether a general finite dimensional algebra is quasi-hereditary by showing that the question can be reduced to showing that a closely associated monomial algebra is quasi-hereditary. For monomial algebras, we give an explicit, easily verifiable, necessary and sufficient criterion to determine whether it is quasi-hereditary.  
\end{abstract}
\date{\today}
\maketitle

\section{Introduction}

Quasi-hereditary algebras are ubiquituous in representation theory. For example, they appear as Schur algebras in the representation theory of the symmetric group, as $q$-Schur algebras in the context of finite groups of Lie type, and more generally  hereditary algebras and algebras of global dimension two are quasi-hereditary.
Quasi-hereditary algebras were first introduced in \cite{Sc}, followed by a detailed study in the context of  highest weight categories and the representation theory of Lie algebras  in \cite{CPS}. There it is shown that every highest weight category with finite weight poset and all objects of finite lengths is the category of finite dimensional modules over a quasi-hereditary algebra. Furthermore, it is shown that a finite dimensional algebra is quasi-hereditary if and only if its module category is a highest weight category. 

Highest weight categories play an important role in many areas of mathematics. For example, although only formally defined in \cite{CPS}, they connect representation theory with geometry in the work of Beilinson and Bernstein \cite{BB} and Brylinski and Kashiwara \cite{BK} linking perverse sheaves and highest weight representation theory  in their proof of the Kazhdan-Lusztig conjecture.  The work of Dlab and Ringel \cite{DR} provides a more algebraic ring theoretic approach to quasi-hereditary algebras connecting it in \cite{DR2} to highest weight categories.

Every finite dimensional algebra  over an algebraically closed field is Morita equivalent to a quotient of a path algebra by an admissible ideal and every such quotient of a path algebra has an associated monomial algebra (see Section 2 for details).  A monomial algebra, is a quotient of a path algebra by an ideal generated by paths (as opposed to linear combinations of paths). For example, in the case of a quotient of a free algebra, a monomial algebra is given by a quotient by an ideal generated by (non-commutative) monomials. 
The properties of the original algebra are closely linked to the properties of the associated monomial algebra, see for example recent work \cite{CS, GHS}. 

While quasi-hereditary algebras are an important class of algebras, establishing whether  an algebra is quasi-hereditary or not is, in general, a difficult problem.   In this paper we give an explicit necessary and sufficient criterion to determine whether a monomial algebra is quasi-hereditary. Based on this result, we give a sufficient criterion to determine whether a general finite dimensional algebra is quasi-hereditary. 
Beginning with \cite{DR} in the early 1990s, a purely algebraic approach to the study of quasi-hereditary algebras has evolved as an active area of research with exciting new results appearing in recent years, for example, \cite{BLvdB, HP, I, IR,  RvdB, R}.

The structure of the paper is as follows. We begin  by briefly recalling  the definition of quasi-hereditary algebras and the basic facts of non-commutative \grb basis theory  including the notion of the associated monomial algebra in Section 2. 
In Section 3 we present an easily verifiable algorithm to determine whether or not a monomial algebra is quasi-hereditary. In Section 4,  we  prove that an algebra is quasi-hereditary  if its associated monomial algebra is quasi-hereditary.

{\bf Acknowledgments:} We thank the referee for helpful comments and suggestions.

\section{Notations and summary of known results}

 Throughout this paper we assume that $K$ is a field,  $Q$ is a quiver, and that, unless otherwise stated,  every $K$-algebra of the form $KQ/I$ is such that $I$ is an admissible ideal in $KQ$.  For a $K$-algebra $\Lambda$,  denote by $J(\Lambda)$ the Jacobson radical of $\Lambda$.   Unless otherwise stated, all modules are finitely generated right $\Lambda$-modules. 

 \subsection{Quasi-hereditary algebras} We begin by recalling some  definitions and results on quasi-hereditary algebras. 

A two-sided ideal $L$ in $\Lambda$ is called {\it heredity} if 
\begin{enumerate}
\item $L^2 = L $
\item  $ L J(\Lambda) L = 0$
\item $L$ is projective as a left or right $\Lambda$-module.
\end{enumerate} 

An algebra $\Lambda$ is {\it quasi-hereditary} if there exists a chain of two-sided ideals 
\begin{equation}\label{heredity chain}0 = L_0 \subset L_{1} \subset \ldots \subset L_{i-1} \subset L_i \subset \ldots \subset L_n = \Lambda\end{equation}
such that $L_i / L_{i-1} $ is a heredity ideal in $\Lambda / L_{i-1} $, for all $i$. We call the sequence in (\ref{heredity chain}) a {\it heredity chain} for $\Lambda$. 

Recall that the {\it trace ideal of a $\Lambda$-module $M$} in $\Lambda$ is the two-sided ideal generated by the images of homomorphisms of $M$ in $\Lambda$. By \cite{APT}, a two-sided ideal $\ma$ in $\Lambda = KQ/I$ is such that $\ma^2 = \ma$ if and only if $\ma$ is the trace ideal of a projective $\Lambda$-module $P$ in $\Lambda$.   There exists a
set   $S$ of distinct vertices in   $Q_0$ such that 
we may assume that $P = \sum_{v \in S} v \Lambda$. 
Hence the trace of $P$ in $\Lambda$ is the two-sided ideal generated by $S$, that is, it is the ideal $\Lambda e \Lambda$ where $e = \sum_{v \in S} v$. See also Statement (6) in \cite{DR}.

Therefore $\Lambda = KQ/I$ is quasi-hereditary 
 if there  exists a sequence of idempotents $e_1, \ldots, e_n$  in $\Lambda$ such that

$$0  \subset \Lambda e_1 \Lambda \subset \ldots \subset \Lambda e_{i-1} \Lambda \subset \Lambda e_i  \Lambda \subset \ldots \subset   \Lambda e_n  \Lambda = \Lambda$$
and such that $\Lambda e_i \Lambda  / \Lambda e_{i-1}  \Lambda$ is a heredity ideal in $\Lambda / \Lambda e_{i-1} \Lambda $, for all $i$.

Note that by Statement (7) in \cite{DR}, for $\Lambda e \Lambda$ where $e$ is an idempotent  in $\Lambda$ such that $ e J(\Lambda) e = 0$, 
 $\Lambda e \Lambda$ is projective as a left $\Lambda$-module  if and only if the map $ \Lambda e \otimes_{e \Lambda e} e \Lambda \to \Lambda e \Lambda$ is bijective. 

\subsection{Non-commutative \grb  basis theory} We now  recall  some  non-commutative \grb basis theory for path algebras needed for our results. Note that in this subsection only, the ideals considered need not be admissible.  For more details on non-commutative \grb basis theory see \cite[Chapter 2]{G1} or the more detailed summary in Sections 2 and 3 in  \cite{GHS}.

 Denote by   $\cB$ the set of finite (directed) paths in a quiver $Q$.  We view the vertices of $Q$ as paths of length zero and so they are elements in $\cB$. Thus  $\cB$ is a $K$-basis for $KQ$.

A nonzero element $x \in KQ$ is \emph{uniform} if $vxw=x$ for  vertices $v, w \in Q_0$. We set  $v = \mathfrak o(x)$ and call it the \emph{origin vertex} of $x$. Similarly,  we set $w= \mathfrak e(x)$, and
call it the \emph{end vertex} of $x$. Since $1=\sum_{v\in Q_0}v$, every
nonzero element of $KQ$ is a sum of uniform elements; namely
$x=\sum_{u,v\in Q_0}uxv$.

To define a  Gr\"obner basis theory,  we need  the notion of an admissible order on $\cB$. 

\begin{defn}
An \emph{admissible order} $\succ$ on $\cB$ is a well-order on $\cB$ that preserves non-zero left and right multiplication and such that if  $p = rqs$, for  $p,q,r,s  \in \cB$  then $p  \succeq q$.
\end{defn}

Recall that a well-order is a total order such that every non-empty subset has a minimal element. 

An example of an admissible order is the left or right length-lexicographical order. 

We now fix an admissible order $\succ$ on $\cB$. 
The order $\succ$ enables us to find the largest path occuring in an element in $KQ$.
 We also can find the largest paths  that occur in elements of a subset of $KQ$.

\begin{defn} If $ x = \sum_{ p \in \cB} \alpha_p p$, with $\alpha_p\in K$, almost all $\alpha_p=0$,  and $x \neq 0$ then define the {\it tip of $x$} to be

\begin{center} $\tip (x) = p$ if $\alpha_p\ne 0$ and $p \succeq q$ for all $q$ with $ \alpha_q\ne 0$. \end{center}

If $X \subseteq KQ$ then $$\tip(X) = \{ \tip(x) | x \in X \setminus \{0\} \}.$$
\end{defn}

 For $p,q\in \cB$, we say that $p$ is a subpath of $q$ and write $p| q$,   if there exist $r,s\in\cB$ such that
$q=rps$. For $A \subset KQ$, we denote by $\langle A \rangle$  the ideal generated by $A$.

 \begin{defn}
 Let $I$ be an ideal in $KQ$.  A subset $\cG$ of uniform elements  in $I$ is a \emph{Gr\"obner basis} for $I$ (with respect to $\succ$) if  $$\langle \tip (I) \rangle = \langle \tip (\cG) \rangle. $$ 
 \end{defn}

 Equivalently, a subset $\cG$ of $I$
consisting of uniform elements, is a \grb basis for $I$ with respect to $\succ$ if, for 
every $x \in I$, $ x \neq 0$, there exists a $g \in \cG$ such that $\tip (g)\vert\tip(x)$.

 It can be shown that 

\begin{proposition}\cite[Proposition 2.4]{GHS} If $\cG$ is a \grb basis for  an ideal $I$ of $KQ$, then $\cG$ is a generating set for $I$;  that is,
$$ \langle \cG \rangle = I.$$
\end{proposition}

 An ideal in $KQ$ is   called  \emph{monomial} if it can be generated by monomials. Recall that the \emph{monomial} elements  in $KQ$ are the elements of $\cB$. 
  If $\Lambda=
KQ/I$ and $I$ is a monomial ideal, then we say that $\Lambda$ is 
a \emph{monomial algebra}.  We recall the following well-known facts about monomial ideals. 
 
 \begin{proposition}\cite[Proposition 2.4]{G1}\label{prop-mono} Let $I$ be a monomial ideal in $KQ$. Then 
 
 (1) there is a unique minimal set $\cT$ of monomial generators for $I$ and $\cT$ is a \grb basis for $I$ for any admissible order on $\cB$. 

 (2)  if $x = \sum_{ p \in \cB} \alpha_p p$ then $x \in I$ if and only if $p \in I$ for all $p$ such that $\alpha_p \neq 0$. 
 \end{proposition}
 
 Let $I$ be an admissible ideal in $KQ$. Set $\cN = \cB \setminus \tip (I)$.  Note that it follows that we also have  $  \cN =\cB\setminus \langle \tip(I) \rangle$: 
  Clearly $\cB\setminus  \langle \tip(I) \rangle  $ is a subset of $ \cB\setminus \tip(I)$.
Now let $p$ be in $\cB\setminus \tip(I)$.    Suppose $p  \in  \langle \tip(I) \rangle$. Then there
are paths $ s_i, u_i \in \cB$ and $t_i \in \tip( I)$ such that $p=\tip(\sum_i s_it_iu_i)$.  Hence
$p=s_it_iu_i$ for some $i$.  There exists $x\in I$ such that $\tip(x)=t_i$.
Hence $s_ixu_i\in I$ and $\tip(s_ixu_i)=\tip(s_it_iu_i)=p$. Thus $p \in \tip(I)$,  a contradiction.
In the following let $\cT$ be the minimal set of monomials that generates the 
monomial ideal {$\langle \tip (I) \rangle$.}
 It follows from the above that we have   $$ \cN = \{ p \in \cB \mid t \mbox{ is not a subpath of $p$ for all $t \in \cT$}\}.$$
 
 The following result is of central importance.
 
 \begin{lemma} \rm{ \bf{(Fundamental Lemma)}} \cite[Paragraph after Definition 2.4]{G1}
 Let $I$ be an ideal in $KQ$. Then  there is a  $K$-vector space isomorphism 
 $$ KQ \simeq I \oplus \mbox{ Span}_K \cN.$$
 \end{lemma}

It is an immediate consequence of the Fundamental Lemma that if $x \in KQ \setminus \{ 0\}$ then $x = i_x + n_x$ for a unique $i_x \in I$ and a unique $n_x \in \mbox{ Span}_K \cN$.
 
  Let $\pi\colon KQ \to KQ/I$ be the canonical surjection. Then  the map  $\sigma\colon  KQ/I \to KQ$ given by $\sigma \pi (x) = n_x$ for $x \in KQ$  is a $K$-vector space splitting of $\pi$.
  We see  that $\sigma$ is
well-defined since if $x,y\in KQ$ are such that $\pi(x)=\pi(y)$, then $x-y\in I$.  Hence
$n_{x- y} =n_x-n_y=0$ and we conclude  that $n_x=n_y$.  Thus,  restricting to  $\mbox{ Span}_K (\cN)$,  we have inverse $K$-isomorphisms $\pi : \mbox{ Span}_K (\cN) \to KQ/I$ and $\sigma : KQ/I \to \mbox{ Span}_K (\cN)$. 
Therefore,  as vector spaces, we can identify $KQ/I$ with $\mbox{ Span}_K \cN$.  We note that  for $x, y \in \mbox{ Span}_K \cN$, the multiplication of $x$ and $y$ in $KQ/I$ equals $n_{x \cdot y}$ where  $x \cdot y $ is the usual multiplication in $KQ$. 

 Summarising, we have the following useful characterisation of a basis of $KQ/I$.
 
 \begin{proposition}\label{Basis N} As $K$-vector spaces, $\Span_K(\cN)$ is isomorphic to $KQ/I$ and hence $\cN$ can be identified with  a $K$-basis of $KQ/I$.
 \end{proposition}

\begin{defn}
We call $I_{Mon} = \langle \tip(I) \rangle$  the ideal   in $KQ$ generated by $\tip(I)$ and define the \emph{associated monomial algebra of $\Lambda=KQ/I$} to be $\Lambda_{Mon} =KQ / I_{Mon}$. 
\end{defn} 

Recall that  $I_{Mon} $ is a monomial ideal, and that by Proposition \ref{prop-mono}(1), there is a unique
minimal set $\cT$ of paths that generate $I_{Mon}$.  
By the Fundamental Lemma there exist unique elments $g_t\in I$ and $n_t\in\Span_K(\cN)$, such that
$t=g_t+n_t$, for $t\in\cT$.  Since $t$ is uniform, $g_t$  and $n_t$ are uniform. Furthermore, since $n_t\in\Span_K(\cN)$, we have that 
$\tip(g_t)=t$ . We now set $\cG=\{g_t\mid t\in\cT \} \subset I$.   Then $\tip(\cG)=\cT$ and hence
$\cG$ is a \grb basis for $I$ (with respect to any admissible order) since $\langle \tip(\cG)\rangle=\langle \cT\rangle =\langle\Imo\rangle=\langle
\tip(I)\rangle$.

\begin{defn}  Let $I$ be an ideal in $KQ$. The set $\cG=\{g_t\mid t\in\cT \} \subset I$ defined above is called the \emph{reduced
\grb basis for $I$ (with respect to $\succ$)}.
\end{defn}

 The next result lists some facts about reduced \grb bases and the associated
monomial algebras, which will be useful for the proofs later in the paper.

\begin{proposition}\label{prop-redgb} Let $I$ be an ideal in  $KQ$
and let $\Lambda=KQ/I$.  Let $\cT$ be the unique minimal set of monomials generating
$\langle \tip(I)\rangle$ and let $\cG$ be the reduced \grb basis for $I$. Then the following hold.
\begin{enumerate}
\item  The reduced \grb basis for $\Imo$ is $\cT$.
\item  $\dim_K(\Lambda)=\dim_K(\Lambda_{Mon}) = \vert \cN \vert $  where $\vert \cN \vert$ denotes the cardinality of the set $\cN$. 
\end{enumerate}
\end{proposition}

Keeping the notation above, we write elements of both $\Lambda$ and $\Lambda_{Mon}$
as $K$-linear combinations of elements in $\cN$.   The difference is how these elements
multiply in $\Lambda$ and $\Lambda_{Mon}$.  Next  assume that $I$ is an admissible ideal.  Notice that $I$  admissible implies that $I_{Mon}$ is admissible. 
 The converse is false in general. If $I$ is admissible,  then the set of vertices and arrows
are always in $\cN$ and both $\cT$ and $\cN$ are finite sets.

\section{Quasi-hereditary monomial algebras}\label{sec-monomial}

In this section we give necessary and sufficient conditions for a monomial algebra to be quasi-hereditary. We also describe  the
structure of an algebra of the form $\Lambda/\Lambda e \Lambda$. 

We fix the following notation for the remainder of
this paper: $K$ will denote a field, $Q$ a quiver, $\cB$ the
set of paths in $KQ$, $\succ$ an admissible order on $\cB$,
$v_1,\dots,v_r$ a set of distinct vertices in $Q$, and
$e=\sum_{i=1}^rv_i$.  We fix an admissible ideal $I$ in $KQ $, let
$\Lambda=KQ/I$, and $J(\Lambda)$ be the Jacobson radical of $\Lambda$.

 \begin{definition} {\rm We say that a vertex {\it $v$ is properly internal} to a path $p \in \cB$, if there exist $p_1, p_2 \in \cB$ both of length greater than or equal to 1 and $p = p_1 v p_2$.  \\
If $\cT$ is a set of paths,
then a 	\emph{ vertex $v$ is not properly internal to $\cT$}
if, for each $t\in\cT$, $v$ can only occur in $t$ as either
the origin or end vertex of $t$.}
\end{definition}

The next result shows the importance of this definition.

\begin{proposition}\label{internal}
Let $\Lambda=KQ/I$  be a finite dimensional monomial algebra and let $\cT$ be a minimal set of generators of paths of $I$.   For  $v \in Q_0$,  the ideal  $\Lambda v \Lambda$  is heredity if and only if $v$ is 
not properly internal to  $\cT$.
\end{proposition}

\begin{proof} First  assume that $v$ is not properly internal to any path in $\cT$. We begin by showing that $v J(\Lambda) v = 0$. Suppose that $v J(\Lambda) v \neq 0$. Then there exists a path $p$ in $J(\Lambda)$ such that $vpv \notin I$. Since $\Lambda$ is finite dimensional, $(vpv)^n \in I$ for sufficiently large $n$. Since $\cT$ is a \grb basis for $I$, there is some $t \in \cT$ such that $t$ is a subpath of $(vpv)^n$. But $ t $ is not a subpath of $p$ and hence $v$ must be internal to $t$ which is a contradiction. Now consider the multiplication map $\mu: \Lambda v \otimes_{v\Lambda v} v\Lambda \to \Lambda v \Lambda$. It suffices to show that this map is bijective \cite{DR}. The map is clearly onto. Since $v J(\Lambda) v = 0$, we have that $n \otimes n'$, for $n\in \cN v$ and $n' \in v\cN$ form a  basis of $\Lambda v \otimes_{v\Lambda v} v\Lambda$. Suppose that $\sum_{n_i \in \cN v, n'_i \in v \cN} n_i v n'_i  \in I$. Then there exists $t \in \cT$ such that $t$ is a subpath of $n_i vn'_i$ for some $i$. But $t$ is not a subpath of $n_i$ or $n'_i$ therefore $v$ is internal to $t$ which is a contradiction.  Hence $\mu$ is bijective and  $\Lambda v
\Lambda$ is a heredity  ideal in $\Lambda$.

Now assume $v$ is internal to some $t \in \cT$. If $v J(\Lambda) v \neq 0$ then $\Lambda v\Lambda$ is not a heredity ideal. Now assume that $v J(\Lambda) v = 0$. Since $v$ is internal to $t$, $t = t_1 v t_2$ with the length of $t_1$ and $t_2$ being at least 1 and  $t_1$ and $t_2$ are not in $I$.  But $t_1 v \otimes v t_2$ is nonzero (since $t_1v\in v\cN$ and
$vt_2\in v\cN$) and maps to $\mu(t_1v \otimes v t_2)=0$ since
$t \in I$. Therefore $\mu$ is not injective.   Hence $\Lambda v\Lambda$ is
not a heredity  ideal in $\Lambda$.
\end{proof}

Note that the first part  of the proof of Proposition~\ref{internal} gives a criterion for  $\Lambda v\Lambda$ to be a heredity ideal 
for not only monomial algebras but for any $\Lambda = KQ/I$.

\begin{corollary} \label{remark32}
Let $\Lambda = KQ/I$, $\cG$ be the reduced \grb basis for $I$ and let $\cT = \tip (\cG)$. Suppose $v$ is vertex in $Q$. 
Then $\Lambda v \Lambda$ is a heredity ideal  in $\Lambda$  if $v$ is not properly internal to any path in $\cT$.
\end{corollary}

 For $e = v_1 + \cdots + v_r$, we define $Q\whe$ to be the subquiver of $Q$ obtained  by
removing the vertices $v_1,\ldots, v_r$ and all arrows entering
or leaving any of the $v_i$'s.  Let $\cB\whe$  be the set
of paths in $KQ\whe$.  Note that the admissible order $\succ$
restricts to an admissible order on $\cB\whe$. We leave the proof of the
next result to the reader

\begin{proposition}\label{prop-decomp}
As $K$-vector spaces,
\[KQ =KQ\whe\oplus \cE.\]
\qed
\end{proposition}

Based on the fundamental Lemma, we fix the following notation: if  $x\in KQ$, we write $x=x\whe +x_e$  for the unique elements
$x\whe\in KQ\whe$ and $x_e\in\cE$, and if  $X\subseteq KQ$,
then define $X\whe=\{x\whe\mid x\in X\}$.

We list some of the basic results that relate these definitions and
leave the proofs to the reader.

\begin{proposition}\label{prop-relate}
{\phantom x}

\begin{enumerate}
\item If $x\in KQ$, then $(x\whe)\whe=x\whe$ and $(x_e)_e=x_e$.
\item If $I$ is an ideal in $KQ$, then $I\whe$ is an ideal
in $KQ\whe$.
\item  $(KQ)\whe=K(Q\whe)$.
\item If $X\subseteq KQ$, then $X\whe=\{x\in KQ\whe\mid \exists x'\in\cE \text{ such that }
x+x'\in X\}$
\end{enumerate}\qed
\end{proposition}

Let $\iota\colon KQ\whe\to KQ$ be the inclusion map and
$\pi\colon KQ\to \Lambda$ and $\rho\colon \Lambda \to
\Lambda/\Lambda e \Lambda$ be the canonical surjections.
Because of its usefulness, we include a proof of the following
well known result.

\begin{lemma}\label{lem-ker}  Let $\Lambda = KQ/I$ be finite dimensional and for $v_1, \ldots, v_r \in Q_0$, set $e = v_1 + \cdots v_r $.  Let 
\[\varphi\colon
KQ\whe\extto{\iota}KQ\extto{\pi}\Lambda \extto{\rho}\Lambda/\Lambda e \Lambda.\]
Then $\varphi$ is surjective and $\ker(\varphi)=
I\whe$.
\end{lemma}

 \begin{proof}We begin by showing $\varphi$ is surjective.  Let $\bar \lambda\in \Lambda/\Lambda e
\Lambda$ and $\lambda\in \Lambda$ such that $\rho(\lambda)=\bar \lambda$.  Let $r\in KQ$ map to $\lambda$ under the  canonical
surjection $\pi$.
Then $r=r\whe +r_e$  and  $\varphi(r\whe)=\rho\pi(r\whe)=\rho(\pi(r\whe + r_e))$ since
$\rho(\pi(r_e))=0$.  We have
$\varphi(r\whe)=\rho(\pi(r))=\bar \lambda$ and
we conclude that $\varphi$ is surjective.

Now we prove that $\ker \varphi = I\whe$.  We start by showing that $I\whe \subset \ker \varphi$. For that  let  $x \in I$. 
 We have that.  
$ x= x\whe + x_e\in I$   and hence 
 $x\whe \in KQ_{\widehat{e}}$ such that there exists $x' \in\cE$ with $x\whe +x' \in I$.  By Proposition \ref{prop-relate},
$x\whe\in I\whe$.  
On the other hand,
let $x \in \ker \varphi$. Then $\pi (\iota(x)) \in \Lambda e \Lambda$ and therefore 
$\pi (\iota(x)) = \sum_i \pi (y_i) e \pi (z_i)$ for some $y_i, z_i \in KQ$. Set $-x ' = \sum_i
y_i e z_i$. Then $x + x' \in I$. Hence $x\in I\whe$. 
\end{proof}

As an immediate  consequence we have the following corollary.

\begin{corollary}\label{alg-iso}There is natural isomorphism of
algebras
$$ \Lambda / \Lambda e \Lambda \simeq KQ_{\widehat{e}} / I_{\widehat{e}}. $$\qed
\end{corollary}

\begin{remark}\label{monomial hat ideal}
{\rm If $I$ is  a monomial ideal  with minimal generating set of paths $\cT$, then $I_{\widehat{e}} =\langle \cT \rangle\whe= \langle \cT_{\widehat{e}} \rangle$. In particular, if $I$ is a monomial
ideal in $KQ$, then $I\whe$ is a monomial ideal in $KQ\whe$.}
\end{remark}

Recall the following well-known Lemma which follows directly from \cite{APT}, which shows that in our context it is enough to consider idempotents of the form $\sum v_i$, 
where the $v_i$ are some vertices in $Q_0$.

\begin{lemma}\label{lem-idempotent}
Let $e$ be an idempotent in $\Lambda = KQ/I$, $I$ admissible. Then $\Lambda e \Lambda = \Lambda \sum_i v_i \Lambda$  for  some $v_i \in Q_0$. 
\end{lemma}

We  can now give the  following necessary and sufficient conditions for a monomial algebra to be quasi-hereditary. 

\begin{theorem}\label{thm vertex ordering}
 Let $\Lambda=KQ/I$  be a finite dimensional monomial algebra, where I  admissible and let $\cT$ be  the minimal set of generators of paths of $I$. Then $\Lambda $   is quasi-hereditary  if and only if we can order all the vertices $v_1, \ldots, v_n$ in $Q_0$ such that 
for each $i$, the vertex $v_i$ is not properly internal to  $\cT_{\wh{i-1}}$. \end{theorem}

\begin{proof}
First assume that we can order all the vertices $v_1, \ldots, v_n$ in $Q_0$ such that 
for each $i$, the vertex $v_i$ is not properly internal to  $\cT_{\wh{i-1}}$. It suffices to show that 
$\Lambda e_i \Lambda / \Lambda e_{i-1} \Lambda$ is a heredity ideal in 
$\Lambda / \Lambda e_{i-1} \Lambda$ for all $i$. Since $v_1$ is not properly internal in $\cT$, by  Corollary \ref{remark32}  we have that $\Lambda v_1 \Lambda$ is heredity in $\Lambda$. 
It follows from Corollary~\ref{alg-iso} that $\Lambda / \Lambda v_1 \Lambda \simeq KQ_{\widehat{v_1}}  /
\langle \cT_{\widehat{v_1}} \rangle $ since $\langle \cT_{\widehat{v_1}} \rangle=\langle \cT\rangle_{\widehat{v_1}}$.
We proceed by induction on $i$. Assume that we have shown the result for $i$. That is $\Lambda e_{i} \Lambda / \Lambda e_{i-1} \Lambda$ is a heredity ideal in 
$\Lambda / \Lambda e_{i-1} \Lambda$.  

We wish to show that  $\Lambda e_{i+1} \Lambda / \Lambda e_{i} \Lambda$ is a heredity ideal in 
$\Lambda / \Lambda e_{i} \Lambda$.
By Corollary \ref{alg-iso},    we have that for all $i$, 
\[
\Lambda / \Lambda e_{i} \Lambda\simeq KQ_{\widehat{e_{i}}}/\langle \cT_{\widehat{e_{i}}}\rangle
\]
Consider the following exact commutative diagram:
\[
\xymatrix{
& 0\ar[d]&0\ar[d]\\
0\ar[r]&\Lambda e_i\Lambda\ar[d]\ar[r]&\Lambda\ar^=[d]\ar[r]&KQ_{\widehat{e_i}}/
\langle \cT_{\widehat{e_i}}\rangle\ar^f[d]\ar[r]&0\\
0\ar[r]&\Lambda e_{i+1}\Lambda\ar[d]\ar[r]&\Lambda\ar[d]^{}
\ar[r]&KQ_{\widehat{e_{i+1}}}/
\langle \cT_{\widehat{e_{i+1}}}\rangle\ar[r]\ar[d]&0\\
&\Lambda e_{i+1}\Lambda/\Lambda e_i\Lambda\ar[d]&0&0\\
&0}\]
  By the diagram above, showing that  $\Lambda e_{i+1}\Lambda/\Lambda e_{i}\Lambda$ is
a heredity ideal in $\Lambda/\Lambda e_{i}\Lambda$ is equivalent to showing that $\ker(f)$ is a heredity ideal in
$KQ_{\widehat{e_i}}/\langle \cT_{\widehat{e_i}}\rangle$.
The reader may check that \[KQ_{\widehat{e_{i+1}}}/\langle \cT_{\widehat{e_{i+1}}}\rangle \text{ is isomorphic to }
KQ_{\widehat{e_i}}/
\langle \cT_{\widehat{e_i}}\rangle/(KQ_{\widehat{e_i}}/
\langle \cT_{\widehat{e_i}}\rangle v_{i+1}KQ_{\widehat{e_i}}/
\langle \cT_{\widehat{e_i}}\rangle).\]

Hence $\ker(f)\simeq KQ_{\widehat{e_i}} /
\langle \cT_{\widehat{e_i}}\rangle v_{i+1}KQ_{\widehat{e_i}}/
\langle \cT_{\widehat{e_i}}\rangle$.
By hypothesis, $v_{i+1}$ is not properly internal to  $\cT_{\widehat{e_i}}$,
and so by  Corollary  \ref{remark32}, $\ker(f)$ is a heredity ideal
in $KQ_{\widehat{e_i}}/\langle \cT_{\widehat{e_i}}\rangle$. We
conclude that $\Lambda$ is quasi-hereditary.

Suppose now that $\Lambda$ is quasi-hereditary and let 
\[\label{chain}  0  \subset \Lambda e_1 \Lambda  \subset \cdots \subset \Lambda e_n \Lambda =   \Lambda \]
be a heredity chain for $\Lambda$. 
We proceed by induction on $n$. We begin by showing that if $e_1= v_1 + \cdots + v_m$ then each $v_i$ is not properly internal to $\cT$. Without loss of generality we do this for $v_1$. We have that $v_1 J(\Lambda) v_1 = 0 $ since $e J(\Lambda) e =0$. 
Consider the following diagram  
$$ \xymatrix{  \Lambda v_1 \otimes_{v_1 \Lambda v_1 } v_1 \Lambda  \ar[d]^{h} \ar[r]^{\phantom{xxxxx} f} 
  & \Lambda v_1 \Lambda  \ar[d] \\
\Lambda e_1 \otimes_{e_1 \Lambda e_1 } e_1 \Lambda  \ar[r]^{\phantom{xxxxx}g}  & \Lambda e_1 \Lambda
}$$ where $g$ is a bijection. The map $f$ is clearly onto and $h$ is injective. That implies that $f$ injective and hence bijective. Hence
$\Lambda v_1 \Lambda$ is heredity in $\Lambda = KQ / \langle \cT \rangle$ and by Proposition~\ref{internal}, $v_1 $ is not properly internal to $\cT$. If $n =1$ this proves  the result.
Assume that the result holds for  $i= n-1$. By  the induction hypothesis  $\Lambda  / \Lambda e_i \Lambda$ is quasi-hereditary  with heredity chain 
$$0 \subset \Lambda e_{i+1} \Lambda / \Lambda e_i \Lambda \subset \cdots  \subset \Lambda e_n \Lambda  / \Lambda e_i \Lambda = \Lambda  / \Lambda e_i \Lambda.$$ 
By Corollary~\ref{alg-iso} together with Remark~\ref{monomial hat ideal}, $\Lambda  / \Lambda e_i \Lambda = K Q_{\wh{i}} / \cT_{\wh{i}}$. Since $\Lambda / \Lambda e_i \Lambda $ has a heredity chain of length  $n-i-1$ and $\cT_{\wh{i}} \subset \cT$, we see by induction 
that the result holds.

\end{proof}

\section{Quasi-hereditary algebras}\label{sec-general}

In this Section we show that if an algebra $\Lambda$ is such that   $\Lambda_{Mon}$ is quasi-hereditary then $\Lambda$ is quasi-hereditary.  More precisely, we show

\begin{theorem}\label{thm:implication}
Let $\Lambda = KQ/I$  with $I$ admissible.
If $\Lambda_{Mon}$ is quasi-hereditary then  $\Lambda$ is quasi-hereditary. \end{theorem}

 Before proving Theorem~\ref{thm:implication}, we begin with some preliminary results.

\begin{lemma}\label{lem-import} Let $\cG$ be the reduced
\grb basis for an ideal $I$ in $KQ$ and $\cT=\tip(\cG)$.
Assume that  $v_i$, for  $1\le i\le r $, is not
properly internal to $\cT$ and set $e = v_1 + \cdots + v_r$. Then for any  $g\in\cG$,
 either  $g=g_e\in \cE$, or  $\tip(g)=\tip(g\whe)$.
\end{lemma}

\begin{proof} Let $g \in \cG$ and  let $t=\tip(g)\in\cT$. Suppose that $t\in\cE$. 
There is some $i$ such that $v_i$ occurs in $t$.  Since
$v_i$  is not properly internal to $t$, $v_i$ is either the origin
vertex or end vertex of $t$.  Since $\cG$ is the reduced \grb
basis for $I$, $g$ is uniform and has $v_i$ as the origin or end vertex. Therefore no summand of $g$ is in $KQ_{\whe}$ and $g= g_e$.   Hence, $g\in\cE$ and we are
done.
\end{proof}

 \begin{lemma}\label{grb of Iehat}
  Let $\cG$ be the reduced \grb basis for an ideal $I$ in $KQ$ and let
$\cT=\tip(\cG)$.  Assume that     $v_i$, for  $1\le i\le r $,  is not
properly internal to $\cT$ and set $e = v_1 + \cdots + v_r$. Then $\cG\whe$ is a \grb  basis for $I\whe$ in $KQ\whe$.  
  \end{lemma}

\begin{proof}  Assume that    $v_i$, for  $1\le i\le r $,  is not
properly internal to $\cT$.  We need to show that $\langle \tip (\cG\whe) \rangle =  \langle \tip (I\whe) \rangle$. For this it is enough to show that for any 
 $x\ne 0 \in I\whe$ there is  $g\in\cG$ and paths $p$ and $q$
in $\cB\whe$ such that $\tip(x)=p\tip(g)q$. By  Lemma~\ref{lem-import} it then follows  that $\tip (g) = \tip(g\whe)$, showing that $\cG\whe$ is a \grb basis for $I\whe$.

So let  $x\ne 0 \in I\whe$ and 
consider the set $Y=\{y\in\cE\mid x+y\in I\}$.
By Proposition \ref{prop-relate} (5) the set $Y$ is not empty.
Let $y^*\in Y$ be such that $\tip(y)\succeq \tip(y^*)$
for all $y\in Y$.
We claim that $\tip(x+y^*)=\tip(x)$.

Suppose not; that is, suppose that $\tip(x+y^*)=\tip(y^*)\in\cE$.
Then since $x+y^* \in I$, there is some $g\in\cG$ and paths $p',q'\in\cB$ such that
$$p'\tip(g)q'=\tip(x+y^*) = \tip (y^*) \in\cE.$$  
If either $p'$ or $q'$ is in $\cE$ , then
$p'gq'\in \cE\cap I$.  But then
$y^*-p'gq'\in  Y$ and $\tip(y^*)\succ \tip (y^*-p'gq')$, a
contradiction.   Thus, $p'$ and $q'$ are in $\cB\whe$
and we must have $\tip(g)\in \cE$.   By
Lemma \ref{lem-import} we have that $g\in \cE$.
But we again have $p'gq'\in\cE\cap I$ and
$y^*-p'gq'\in  Y$ and $\tip(y^*)\succ \tip (y^*-p'gq')$, a
contradiction. This proves the claim that
$\tip(x+y^*)=\tip(x)$.

Since $x+y^*\in I$, there is some $g\in\cG$ and paths
$p,q  \in \cB$ such that $p\tip(g)q=\tip(x+y^*)=\tip(x)$.
Since $x\in KQ\whe$,  we have that $\tip(g) = \tip(g\whe)$ and hence $g \in \cG\whe$.  This finishes
the proof that $\cG\whe$ is  a  \grb basis for $I\whe$.
\end{proof}

\begin{proposition}\label{prop-quotient mon}
Let $\cG$ be the reduced \grb basis of $I$ and let
$\cT=\tip(\cG)$.  Assume that  $v_i$, for  $1\le i\le r $,  is not
properly internal to $\cT$ and set $e = v_1 + \cdots + v_r$.
 Then 
$$ (\Lambda / \Lambda e \Lambda)_{Mon} \simeq \Lambda_{Mon} /  \Lambda_{Mon} e \Lambda_{Mon}.$$
\end{proposition}

\begin{proof}  We have that $I_{Mon} = \langle \cT \rangle$ and $\cT = \cT_{\widehat{e}} \cup \cT_e$ where 
$\cT_e =  \cT \cap  \cE.$ 
Furthermore, by Corollary \ref{alg-iso}, we have
$\Lambda_{Mon} /  \Lambda_{Mon} e \Lambda_{Mon} \simeq KQ_{\widehat{e}}/ \langle \cT \rangle_{\widehat{ e}} $ 
and $\Lambda/\Lambda e \Lambda \simeq KQ\whe/I\whe$.
 Now by Lemma~\ref{grb of Iehat},  the set
$\cG\whe$ is a \grb
basis of $I\whe$ and hence $\tip (\cG\whe) = \cT\whe$   is a \grb basis of $(I\whe)_{Mon}$ and the result follows. 
\end{proof}

 Since it is a crucial point in the proof of Theorem~\ref{thm:implication},  we recall the following well-known result.

\begin{lemma}\label{lemma-quot} Let $\Lambda=KQ/I$ and suppose that \[(*)\quad\quad (0)\subset L_1 \subset L_2 \subset \cdots
\subset L_{k-1}\subset L_k=\Lambda\] is a chain of ideals in $\Lambda$.
Then $(*)$ is a heredity chain for $\Lambda$ if and only if
$L_1$ is a heredity ideal in $\Lambda$ and
\[ (0)\subset L_2/L_1 \subset L_3/L_1 \subset\cdots
  \subset L_{k-1}/L_1 \subset L_k/L_1 =\Lambda/L_1\] is a heredity 
chain for $\Lambda/L_1$.
\qed\end{lemma}

{\it Proof of Theorem~\ref{thm:implication}.}
Assume $\Lambda_{Mon}$ is quasi-hereditary and that 
$$0 \subset \Lambda_{Mon} e_1 \Lambda_{Mon} \subset \Lambda_{Mon} e_2 \Lambda_{Mon} \subset \cdots \subset \Lambda_{Mon}$$ 
is a heredity chain for $\Lambda_{Mon}$. By   Theorem~\ref{thm vertex ordering}  we can assume that $e_i = e_{i-1} + v$ for some vertex $v \in Q_0$ and in particular, $e_1 = v_1$. 
We show that 
$$0 \subset \Lambda e_1 \Lambda \subset \Lambda e_2 \Lambda \subset \cdots \subset \Lambda$$ 
is a heredity chain for $\Lambda$. 
We proceed by induction on the number of vertices in $Q$.
If $Q$ has one vertex, the
result is vacuously true.  Assume the result is true for quivers
with $k-1$ vertices and that $Q$ has $k$ vertices and $\Lambda_{Mon}=KQ/(I_{Mon})$ is
quasi-hereditary. 

 Since  $\Lambda_{Mon} v_1\Lambda_{Mon}$ is a hereditary ideal in
$\Lambda_{Mon}$, $v_1$ is not properly internal to $\cT=\tip(\cG)$   where $\cG$ is the reduced \grb basis for $I$. By
Corollary \ref{remark32}, the ideal 
$\Lambda v_1 \Lambda$ is a heredity ideal in $\Lambda$.

We have that $\Lambda_{Mon} $ quasi-hereditary implies $\Lambda_{Mon} / 
\Lambda_{Mon} v_1 \Lambda_{Mon}$
quasi-hereditary and 
by Proposition~\ref{prop-quotient mon}, we have that $ \Lambda_{Mon} / 
\Lambda_{Mon} v_1 \Lambda_{Mon} \simeq (\Lambda / \Lambda v_1 \Lambda)_{Mon} $. By Lemma \ref{lemma-quot},
\[0\subset
\Lambda_{Mon}e_2\Lambda_{Mon}/\Lambda_{Mon}v_1\Lambda_{Mon}\subset\cdots
\subset  \Lambda_{Mon}/\Lambda_{Mon}v_1\Lambda_{Mon} \]
is a heredity chain for $\Lambda_{Mon} /\Lambda_{Mon} v_1\Lambda_{Mon}$.

Since $\Lambda / 
\Lambda v_1 \Lambda$ has strictly fewer vertices than  $\Lambda$, by induction on the number of vertices $\Lambda / 
\Lambda v_1 \Lambda$ is quasi-hereditary with heredity chain 
\[\Lambda e_2\Lambda/\Lambda v_1\Lambda\subset\cdots
\subset \Lambda/\Lambda v_1\Lambda.
\]
Applying Lemma \ref{lemma-quot} again, we conclude that
$\Lambda$ is quasi-hereditary.
\qed  
\vskip .2in

The following example illustrates how the above results
can be applied.

\begin{Example}{\rm

Let $Q$ be the quiver
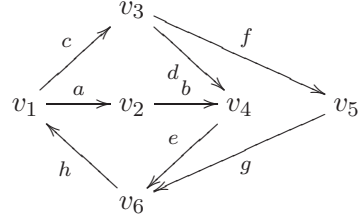
\[\xymatrix{
&v_3\ar[dr]_d\ar[drr]^f\\
v_1\ar[ur]^c\ar[r]^a&v_2\ar[r]^b&v_4\ar[dl]_e&v_5\ar[lld]^g\\
&v_6\ar[lu]^h}
\]
Let $\succ$ be the length-(left)lexicographic order with
$a\succ b\succ c\succ\cdots\succ f\succ g$. Let
$\cT=\{ab, be, de, eh, hc\}$.  First we show that
the  monomial algebra $\Lambda= KQ/\langle\cT\rangle$ is
quasi-hereditary.

 The vertices $v_1, v_2, v_4$ and $v_6$ are properly internal
to $\cT$ and $v_3$ and $v_5$ are  not properly internal to any path in $\cT$.  We choose as first vertex in our vertex ordering $v_3$. Hence,  $\Lambda v_3 \Lambda$ is a heredity  ideal
in $\Lambda$.
By our earlier results $\Lambda /\Lambda v_3\Lambda$
is isomorphic to $KQ_{\widehat v_3}/\cT_{\widehat v_3}$ where $Q_{\widehat v_3}$ is
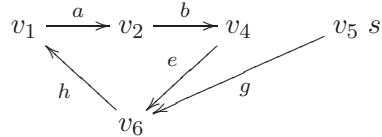
\[\xymatrix{
v_1\ar[r]^a&v_2\ar[r]^b&v_4\ar[dl]_e&v_5\ar[lld]^g\\
&v_6\ar[lu]^h}s
\] and $\cT_{\widehat v_3}=\{ab, be, eh\}$.   Now, for example, vertex
$v_1$ is not properly internal to $\cT_{\widehat v_3}$.  Thus $(\Lambda/\Lambda v_3\Lambda)v_1
(\Lambda/\Lambda v_3\Lambda)$ is a heredity ideal in
$(\Lambda/\Lambda v_3\Lambda)$. Continuing one obtains
a heredity chain
\[ 0\subset \Lambda v_3\Lambda\subset\Lambda v_3+v_1\Lambda\subset\Lambda v_3+v_1+v_2\Lambda\subset
\Lambda v_3+v_1+v_2	+v_4\Lambda\subset
\Lambda v_3+v_1+v_2	+v_4+v_5\Lambda\subset\]\[
\Lambda v_3+v_1+v_2	+v_4+v_5+v_6\Lambda=\Lambda
\]
It follows from Theorem \ref{thm:implication} that if  $\Gamma = KQ/J$ with $J$ admissible and such that  $\Gamma_{Mon} = \Lambda$ then $\Gamma$ is quasi-hereditary.   Using the
results in \cite{GHS}, the  reduced \grb basis is  of the form  $$\{ t - \sum_{t \succ n_i, t || n_i} X_i n_i \mid t \in \cT, X_i \in K \}.$$
Thus in our case we see that $\Gamma$ necessarily is of the form
$\Gamma=KQ/\langle g_1, g_2, \dots,g_5\rangle$ with
$g_1=ab-Xcd, g_2=be, g_3= de - Y fg, g_4=eh, g_5=hc$, and
$X, Y\in K$ are arbitrary. In particular, in this case there are only four isomorphism classes of algebras, namely, if we denote by $\Gamma_{X,Y}$ the algebra with parameters $X$ and $Y$, then representatives of the isomorphism classes can be given by 
$X$ equal $0$ or $1$ and $Y$ equal $ 0$ or $1$. 
}
\end{Example}

\bibliographystyle{plain}

\end{document}